\documentclass[10pt,draft,a4paper]{amsart}
\usepackage{amssymb}
\usepackage{graphicx}
\numberwithin{equation}{section}
\newtheorem{theorem}{Theorem}[section]

\newtheorem{lemma}[theorem]{Lemma}
\theoremstyle{remark}
\newtheorem{remark}{Remark}[section]

\theoremstyle{definition}

\newcommand{\R}{\mathbb{R}}

\newcommand{\N}{\mathbb{N}}

\newcommand{\xx}{\langle x\rangle}

\begin{document}


\title[Counterexamples to Strichartz Estimates]
{Counterexamples to Strichartz estimates for the magnetic Schr\"odinger equation}

\author{Luca Fanelli}
\address{Luca Fanelli:
Universidad del Pais Vasco, Departamento de
Matem$\acute{\text{a}}$ticas, Apartado 644, 48080, Bilbao, Spain}
\email{luca.fanelli@ehu.es}

\author{Andoni Garcia}
\address{Andoni Garcia: Universidad del Pais Vasco, Departamento de
Matem$\acute{\text{a}}$ticas, Apartado 644, 48080, Bilbao, Spain}
\email{andoni.garcia@ehu.es}

\thanks{The second author is supported by the grant BFI06.42 of the Basque
Government}

\begin{abstract}
  In space dimension $n\geq3$, we consider the magnetic Schr\"odinger Hamiltonian $H=-(\nabla-iA(x))^2$ and the corresponding Schr\"odinger equation
  \begin{equation*}
    i\partial_tu+Hu=0.
  \end{equation*}
  We show some explicit examples of potentials $A$, with less than Coulomb decay, for which any solution of this equation cannot satisfy Strichartz estimates, in the whole range of Schr\"odinger admissibility.
\end{abstract}

\date{\today}


\maketitle

\section{Introduction}\label{sec:introd}

Recently, a lot of attention has been devoted to the measurement of the precise decay rate of solutions of dispersive equations. A family of a priori estimates including time-decay, Strichartz, local smoothing and Morawetz estimates are in a sense the core of the linear and nonlinear theory, with immediate applications to local and global well-posedness, scattering and low regularity Cauchy problems. This kind of problems concerns with some of the fundamental dispersive equations in Quantum Mechanics, including among the others the Schr\"odinger, wave, Klein-Gordon and Dirac equations. Strichartz estimates appear in \cite{S}, by R. Strichartz, first; later, the basic framework from the point of view of PDEs was given in the well-known paper \cite{GV} by J. Ginibre and G. Velo, and then completed in \cite{KT} by M. Keel and T. Tao, proving the endpoint estimates.

The Strichartz estimates for the Schr\"odinger equation are the following:
\begin{equation}\label{eq:strifree}
  \left\|e^{it\Delta}f\right\|_{L^p_tL^q_x}\leq C\|f\|_{L^2_x},
\end{equation}
for any couple $(p,q)$ satisfying the Schr\"odinger admissibility condition
\begin{equation}\label{eq:admis}
  \frac2p=\frac n2-\frac nq
  \qquad
  p\geq2,
  \quad
  p\neq2\ \text{if }n=2,
\end{equation}
where $n$ is the space dimension. The importance in the applications leads naturally to consider perturbations of the Schr\"odinger operator $H_0=-\Delta$ by linear lower-order terms; in this paper, we will deal with a purely magnetic Schr\"odinger Hamiltonian $H$ of the form
\begin{equation}\label{eq:hamiltonianintr}
  H=-(\nabla-iA(x))^2,
\end{equation}
where $A:\R^n\to\R^n$ is a magnetic potential, describing the interaction of a free particle with an external magnetic field. The magnetic field $B$, which is the physically measurable quantity, is given by
\begin{equation}\label{eq:B}
  B\in\mathcal M_{n\times n},
  \qquad
  B=DA-(DA)^t,
\end{equation}
i.e. it is the anti-symmetric gradient of the vector field $A$ (or, in geometrical terms, the differential $dA$ of the 1-form which is standardly associated to $A$). In dimension $n=3$ the action of $B$ on vectors is identified with the vector field $\text{curl}A$, namely
\begin{equation}\label{eq:B3}
  Bv=\text{curl}A\times v
  \qquad
  n=3,
\end{equation}
where the cross denotes the vectorial product in $\R^3$.
The investigation on Strichartz estimates for solutions of the magnetic Schr\"odinger equation
\begin{equation}\label{eq:schromagnintr}
  \begin{cases}
    i\partial_tu(t,x)+Hu(t,x)=0
    \\
    u(0,x)=f(x)
  \end{cases}
\end{equation}
is actually in course of development. In the last years, this problem was studied in the papers \cite{DF,EGS1,EGS,GST}, by functional calculus techniques involving Spectral Theorem and resolvent estimates. Later, in \cite{DFVV}, Strichartz estimates, comprehensive of the endpoint, are stated as a consequence of the local smoothing estimates proved in \cite{FV} via integration by parts. The results in the above mentioned papers seem to show that the regularity $|A|\sim|x|^{-1}$, which is the one of the Coulomb potential, should be a threshold for the validity of Strichartz estimates. Indeed, the behavior of $A$, in all those results, is of the type $|A|\sim|x|^{-1-\epsilon}$, as $|x|\to\infty$. Since there is no heuristic, based on scaling arguments, showing that the Coulomb decay is critical for Strichartz estimates, this remains a conjecture until the lack of counterexamples for Strichartz estimates is not overcome.

For electric Schr\"odinger Hamiltonians of the type $-\Delta+V(x)$, some explicit counterexamples are given in the paper \cite{GVV}. In that setting, the critical regularity for the electric potential $V$ is the inverse square one $|V|\sim|x|^{-2}$; the counterexamples in \cite{GVV} are based on potentials of the form
\begin{equation*}
  V(x)=(1+|x|^2)^{-\frac\alpha2}\omega\left(\frac{x}{|x|}\right),
  \qquad
  0<\alpha<2,
\end{equation*}
where $\omega$ is a positive scalar function, homogeneous of degree 0, which has a non-degenerate minimum point $P\in S^{n-1}$. Moreover, it is crucial there to assume that $w(P)=0$. The main idea is to approximate $H$, by a second order Taylor expansion, with an harmonic oscillator. Then the condition $\alpha<2$ causes the lack of global (in time) dispersion.

In this paper, we produce some explicit examples of magnetic potentials $A$, with less than Coulomb decay, for which Strichartz estimates for equation \eqref{eq:schromagnintr} fail. Before stating the main results, we introduce some notations. Let us consider the $2\times2$ anti-symmetric matrix
\begin{equation*}
  \sigma
  :=
  \left(
  \begin{array}{cc}
    0 & 1
    \\
    -1 & 0
  \end{array}\right).
\end{equation*}
For any even $n=2k\in\N$, we denote by $\Omega_n$ the $n\times n$ anti-symmetric matrix generated by $k$-diagonal blocks of $\sigma$, in the following way:
\begin{equation}\label{eq:Omega}
  \Omega_n
  :=
  \left(
  \begin{array}{cccc}
    \sigma & 0 & \cdots & 0
    \\
    0 & \sigma & \cdots & 0
    \\
    \vdots & \vdots & \ddots & \vdots
    \\
    0 & \cdots & 0 & \sigma
  \end{array}
  \right)
\end{equation}
Our first result is for odd space dimensions.
\begin{theorem}\label{thm:odd}
  Let $n\geq3$ be an odd number and let us consider the following anti-symmetric $n\times n$-matrix
  \begin{equation}\label{eq:Modd}
    M
    :=
    \left(
    \begin{array}{cc}
      \Omega_{n-1} & 0
      \\
      0 & 0
    \end{array}
    \right),
  \end{equation}
  where $\Omega_{n-1}$ is the $(n-1)\times(n-1)$-matrix defined in \eqref{eq:Omega}.
  Let $A$ be the following vector-field:
  \begin{equation}\label{eq:A}
    A(x)=|x|^{-\alpha}Mx,
    \qquad
    1<\alpha<2.
  \end{equation}
  Then a solution of the magnetic Schr\"odinger equation \eqref{eq:schromagnintr} cannot satisfy Strichartz estimates, for all the Schr\"odinger admissible couples $(p,q)\neq(\infty,2)$.
\end{theorem}
The analogous result in even dimensions is the following.
\begin{theorem}\label{thm:even}
  Let $n\geq4$ be an even number; let $\mathbf{0}$ be the null $2\times2$-block and let us consider the following anti-symmetric $n\times n$-matrix
  \begin{equation}\label{eq:Meven}
    M
    :=
    \left(
    \begin{array}{cc}
      \Omega_{n-2} & 0
      \\
      0 & \mathbf{0}
    \end{array}
    \right),
  \end{equation}
  where $\Omega_{n-2}$ is the $(n-2)\times(n-2)$-matrix defined in \eqref{eq:Omega}.
  Let $A$ be the following vector field:
  \begin{equation}\label{eq:AA}
    A(x)=|x|^{-\alpha}Mx,
    \qquad
    1<\alpha<2.
  \end{equation}
  Then a solution of the magnetic Schr\"odinger equation \eqref{eq:schromagnintr} cannot satisfy Strichartz estimates, for all the Schr\"odinger admissible couples $(p,q)\neq(\infty,2)$.
\end{theorem}
\begin{remark}\label{rem:l2}
  The $L^\infty L^2$ estimate, which is not disproved by Theorems \ref{thm:odd}, \ref{thm:even} clearly holds as the $L^2$-conservation for solutions of \eqref{eq:schromagnintr}.
\end{remark}
\begin{remark}\label{rem:generalize}
  As will be clear by the proofs of the previous Theorems, for our counterexamples it is sufficient to consider potentials which are not singular at the origin, of the form $A=\xx^{-\alpha}Mx$, where $\xx=(1+|x|^2)^{1/2}$. Indeed, the reason for which Strichartz estimates fail is that these examples, for $1<\alpha<2$, do not decay enough at infinity. It is also possible to generalize the family of potentials producing counterexamples, by considering, instead of the above matrix $\sigma$, the following one
  \begin{equation*}
    \widetilde\sigma
  :=
  \left(
  \begin{array}{cc}
    0 & \omega\left(\frac x{|x|}\right)
    \\
    -\omega\left(\frac x{|x|}\right) & 0
  \end{array}
  \right),  
  \end{equation*}
  where $\omega$ is a scalar function, homogeneous of degree 0, and constructing $\Omega_n$, $M$ and $A$ as above.
\end{remark}
\begin{remark}\label{rem:weak}
  Following the notations in \cite{FV}, we denote by $B_\tau$ the tangential component of $B=DA-(DA)^t$, namely $B_\tau(x)=\frac{x}{|x|}B(x)$.
    It is easy to see that, in the cases \eqref{eq:A}, \eqref{eq:AA}, we have $B_\tau=0$ (see also Examples 1.6 and 1.7 in \cite{FV}). Hence, by Theorems 1.9 and 1.10 in \cite{FV}, weak-dispersive estimates, including local smoothing, hold for equation \eqref{eq:schromagnintr}, independently of the decay rate $\alpha$. In fact, these are relevant examples for which weak dispersion holds, but not Strichartz estimates.
\end{remark}
\begin{remark}\label{rem:alfa}
  The potentials in \eqref{eq:A}, \eqref{eq:AA} have the decay $|A|\sim|x|^{1-\alpha}$ and the statements are given in the range $\alpha\in(1,2)$. In the case $\alpha<1$, since the potential $A$ does not decay at infinity, the spectrum of $H$ contains eigenvalues; hence, due to the presence of $L^2$-eigenfunctions for $H$, global dispersive estimates obviously fail.
  It remains opened the question for $\alpha=2$, for which no results are known, neither of positive or negative type.
\end{remark}
\begin{remark}\label{rem:dimension}
  The counterexamples in the previous Theorems cover the dimensions $n\geq3$. The dimension $n=1$ is not meaningful for magnetic potentials, because by gauge transformations it is always possible to pass from \eqref{eq:hamiltonianintr} to the free Hamiltonian $H_0=-\partial^2/\partial x^2$. On the other hand, the problem of Strichartz estimates for \eqref{eq:schromagnintr}, in dimension $n=2$, remains completely open.
\end{remark}
The idea of the proof of Theorems \ref{thm:odd}, \ref{thm:even} is analogous to the one in \cite{GVV}. By homogeneity and the algebraic form of $A$ in \eqref{eq:A}, \eqref{eq:AA}, it is possible to approximate the Hamiltonian $H$ in \eqref{eq:hamiltonianintr} with the dimensionless operator
\begin{equation}\label{eq:LL}
  T:=-(\nabla-i\Omega y)^2+|y|^2,
\end{equation}
where $\Omega=\Omega_{n-1}$, if $n$ is odd, and $\Omega=\Omega_{n-2}$ if $n$ is even. We recall that the operator
\begin{equation*}
  T_0:=-(\nabla-i\Omega y)^2
\end{equation*}
has compact resolvent and its spectrum is purely discrete. Actually, it is the analogous of the harmonic oscillator, in the setting of magnetic fields, and its eigenvalues define the levels of energy which are usually referred to as the {\it Landau levels} (see e.g. \cite{CFKS}). The operator $T$ in \eqref{eq:LL} has the same spectral properties of $T_0$, since the quadratic form $|y|^2$ is positively defined (see e.g. \cite{M}).

The rest of the paper is organized as follows. In Sections \ref{sec:odd}, \ref{sec:even}, starting from an eigenfunction of $T$ we construct and estimate some approximated solutions to equation \eqref{eq:schromagnintr} (see Lemmas \ref{lem:odd}, \ref{lem:even}). The proof of Theorems \ref{thm:odd}, \ref{thm:even} are performed in Sections \ref{sec:proofodd}, \ref{sec:proofeven}, respectively.

\section{Preliminaries for the odd-dimensional case}\label{sec:odd}

Throughout this Section, $n$ will be an odd number, $n\geq3$.

On the even-dimensional space $\R^{n-1}$, let us consider again the operator $T$ defined in \eqref{eq:LL}. Observe that the explicit expansion of $T$ is the following:
\begin{equation}\label{eq:epo1}
  T:=-\Delta + 2i (y_{2}, -y_{1}, \dots ,y_{n-1}, -y_{n-2}) \cdot \nabla + 2|y|^2,
\end{equation}
for $y\in\R^{n-1}$.
Since the quadratic form $|y|^2$ is positively defined and $\Omega$ is anti-symmetric, $T$ has compact resolvent and in particular its spectrum reduces to a discrete set of eigenvalues (see e.g. \cite{M}). Moreover, any eigenfunction $v(y)$ solving
\begin{equation}\label{eq:epo2}
Tv(y) = \lambda v(y),
\end{equation}
for some eigenvalue $\lambda\in\R$, is such that
\begin{equation}\label{eq:lp}
  v(y) \in \bigcap _{p=1}^{\infty} L^{p}(\R^{n-1}),
  \qquad
  \nabla v \in \bigcap _{p=1}^{\infty} L^{p}(\R^{n-1}).
\end{equation}
Let us fix an eigenvalue $\lambda\in\R$ and a corresponding eigenfunction $v$; by scaling $v$ we create a function $\omega:\R^{n-1}\times(0,+\infty)$ in the following way:
\begin{equation}\label{eq:omega1}
\omega(x) := v\left(\frac{y}{\sqrt{z^{\alpha}}}\right),
\qquad
x:=(y,z) \in \mathbb{R}^{n-1} \times (0, \infty),
\end{equation}
where $\alpha$ is the exponent in \eqref{eq:A}.
By a direct computation, we see that $\omega$ satisfies
\begin{equation}\label{eq:epo4}
 -\left(\nabla-\frac{i}{z^{\alpha}}M(y,0)^t\right)^2\omega +  \frac{1}{z^{2\alpha}}|y|^2\omega = \frac{\lambda}{z^{\alpha}} \omega,
\end{equation}
where $M$ is defined via \eqref{eq:Modd}.
Let us now introduce the time-dependent function
\begin{equation}\label{eq:epo5}
W(t,y,z) = e^{i( \lambda t /z^{\alpha})}\omega(y,z),
\qquad
(t,y,z) \in \mathbb{R} \times \mathbb{R}^{n-1} \times (0, \infty).
\end{equation}
By direct computations, it turns out that $W$ solves
\begin{equation}\label{eq:epo10}
i\partial_{t}W - \left(\nabla - \frac{i}{z^{\alpha}}M(y,0)^t\right)^2 W
+ \frac{1}{z^{2\alpha}}|y|^2W = F,
\end{equation}
where
\begin{align}\label{eq:F}
F(t,y,z)
= & \frac{e^{i( \lambda t /z^{\alpha})}}{z^{2}} \left\{\left[
\frac{\alpha^{2}\lambda^{2}t^{2}}{z^{2\alpha}}-\frac{\alpha(\alpha + 1)i \lambda t}{z^{\alpha}}\right]
v\left(\frac{y}{\sqrt{z^{\alpha}}} \right)
\right.
\\
& \left. - G\left(\frac{y}{\sqrt{z^{\alpha}}}\right)\cdot\left[\frac{\alpha^{2}i \lambda t}{z^{\alpha}} + \frac{\alpha(\alpha + 2)}{4}\right] - \frac{\alpha^{2}}{4} H\left(\frac{y}{\sqrt{z^{\alpha}}}\right)\right\},
\nonumber
\end{align}
with
\begin{equation}\label{eq:epo7}
G(y) = y \cdot \nabla_{y}v(y),
\qquad
H(y) = y D^{2}_{y}v(y)\cdot y.
\end{equation}
We now introduce a real-valued cutoff function $\psi \in C^{\infty}_0(\R)$ with the following properties:
\begin{equation}\label{eq:psi}
\psi(z) = 0\quad \text{for }|z| > 1,
\qquad
\psi(z) = 1\quad \text{for }|z| < 1/2.
\end{equation}
Let us fix a parameter $\gamma \in (1/2,1)$, and for any $R>0$ let us denote by
\begin{equation}\label{eq:psiR}
  \psi_{R}(z) := \psi\left(\frac{z - R}{R^{\gamma}}\right).
\end{equation}
With these notations, we truncate $W$ as follows:
\begin{equation}\label{eq:WR}
 W_{R}(t,y,z) := W(t,y,z)\psi_R(z)\psi\left(\frac{|y|^{2}}{z^{2}}\right).
\end{equation}
Again, a direct computation shows that $W_R$ solves the Cauchy problem
\begin{equation}\label{eq:epo17}
  \begin{cases}
    i\partial_tW_R-(\nabla-\frac{i}{z^\alpha}M(y,0)^t)^2W_{R}
    +\frac{|y|^2}{z^{2\alpha}}W_R
    =
    F_R
    \\
    W_{R}(0,y,z) = f_{R}(y,z).
  \end{cases}
\end{equation}
Here the initial datum is given by
\begin{equation}\label{eq:epo18}
f_{R}(y,z) = \psi_{R}(z)\psi\left(\frac{|y|^{2}}{z^{2}}\right)\omega(y,z),
\end{equation}
and
\begin{equation}\label{eq:epo19}
F_{R}(t,y,z) =  \psi_{R}\psi F + G_{R},
\end{equation}
where $F$ is given by \eqref{eq:F} and
$G_R$ has the form
\begin{align}\label{eq:G}
G_{R}(t,y,z) = &e^{i( \lambda t /z^{\alpha})}\left\{-\omega\left[\frac{2(n - 1)}{z^{2}}\psi_{R}\psi' + \frac{4|y|^{2}}{z^{4}}\psi_{R}\psi'' + \psi''_{R}\psi - \frac{4|y|^{2}}{z^{3}}\psi'_{R}\psi'
\right.\right.
\\
&\left.+ \frac{4|y|^{4}}{z^{6}}\psi_{R}\psi'' + \frac{6|y|^{2}}{z^{4}}\psi_{R}\psi' - \frac{2\alpha i\lambda t}{z^{\alpha+1}}\psi'_{R}\psi + \frac{4\alpha i\lambda t|y|^{2}}{z^{\alpha+4}}\psi_{R}\psi'\right]
\nonumber
\\
& \left.- G\left(\frac{y}{\sqrt{z^\alpha}}\right)\cdot\left[ \frac{4}{z^{2}}\psi_{R}\psi' - \frac{\alpha}{z}\psi^{'}_{R}\psi + \frac{2\alpha |y|^{2}}{z^{4}}\psi_{R}\psi'\right]\right\},
\nonumber
\end{align}
and $G$ is defined by \eqref{eq:epo7}.

The main result of this Section is the following.
\begin{lemma}\label{lem:odd}
Let $n\geq 3$ be an odd number, $p,q\in(1,\infty)$, and $\gamma\in(1/2,1)$; then we have
\begin{equation}\label{eq:epo21}
 \|f_{R}\|_{L^{2}_{x}} \leq C R^{(\alpha(n-1)+2\gamma)/4},
\end{equation}
\begin{equation}\label{eq:epo22}
 \|W_{R}\|_{L^{p}_{T}L^{q}_{x}} \geq C T^{1/p} R^{(\alpha(n-1)+2\gamma)/2q},
\end{equation}
\begin{equation}\label{eq:epo23}
 \|F_{R}\|_{L^{p}_{T}L^{q}_{x}} \leq C T^{1/p} R^{(\alpha(n-1)+2\gamma)/2q}\max\{R^{-2\gamma}, T^{2}R^{-(2\alpha+2)}\},
\end{equation}
for all $R > 2, T > 0$, and some constant
where $C = C(q,\gamma) > 0$. In particular, if $(p,q)\neq(\infty,2)$ is a Schr\"odinger admissible couple
and $\beta > 0$, then the following estimates hold
\begin{equation}\label{eq:epo25}
\frac{\|W_{R}\|_{L^{p}((0,R^{\beta}); L^{q}_{x})}}{\|f_{R}\|_{L^{2}_{x}}} \geq CR^{(\beta n - (\alpha(n-1)+2\gamma))/np},
\end{equation}
\begin{equation}\label{eq:epo26}
\frac{\|W_{R}\|_{L^{p}((0,R^{\beta}); L^{q}_{x})}}{\|F_{R}\|_{L^{p'}((0,R^{\beta}); L^{q'}_{x})}} \geq CR^{\kappa},
\end{equation}
for any $R>2$,
where
\begin{equation*}
\kappa = \kappa(n, \gamma, \beta, p) = 2\left(\frac{\beta n - (\alpha(n-1)+2\gamma)}{np}\right) +\min\{2\gamma - \beta, 2\alpha + 2 - 3\beta\}
\end{equation*}
and the constants $C > 0$ do not depend on $R$.
\end{lemma}
\begin{proof}
For a given function $L(y)$, we denote by
\begin{equation}\label{eq:epo28}
\Lambda(y,z) := L\left(\frac{y}{\sqrt{z^{\alpha}}}\right),
\end{equation}
for any $(y,z) \in \mathbb{R}^{n-1} \times \mathbb{R}.$
In order to prove \eqref{eq:epo21} and \eqref{eq:epo22}, it is sufficient to show that, if $0\neq L\in \bigcap _{p=1}^{\infty} L^{p}(\mathbb{R}^{n-1})$, then the following estimates hold:
\begin{equation}\label{eq:epo29}
cR^{(\alpha(n-1)+2\gamma)/2q} \leq  \|\Lambda\psi_{R}\psi\|_{L^{q}_{x}} \leq CR^{(\alpha(n-1)+2\gamma)/2q},
\end{equation}
for all $R>1$,
where $c = c(q,L) > 0$ and $C = C(q,L) > 0$ and $\psi$ and $\psi_{R}$ are defined by \eqref{eq:psi}, \eqref{eq:psiR}. Indeed, \eqref{eq:epo21} and \eqref{eq:epo22} follow by \eqref{eq:epo29}, with the choice $L(y)=v(y)$. In order to prove \eqref{eq:epo29}, observe that, by the properties of $\psi_{R}$ and $\psi$, we have
\begin{align*}
& \int_{R-R^{\gamma}/2}^{R+R^{\gamma}/2}dz\int_{|y| < \frac{\sqrt{z^{2-\alpha}}}{\sqrt{2}}}|L(y)|^{q}z^{(n-1)\alpha/2}dy
\\
&\ \ \
= \int_{R-R^{\gamma}/2}^{R+R^{\gamma}/2}dz\int_{|y| < \frac{z}{\sqrt{2}}}|\Lambda|^{q}dy
\leq \int_{\mathbb{R}^{n}}|\Lambda\psi_{R}\psi|^{q}dydz
\leq \int_{R-R^{\gamma}}^{R+R^{\gamma}}dz\int_{|y| < z}|\Lambda|^{q}dy
\\
&\ \ \
= \int_{R-R^{\gamma}}^{R+R^{\gamma}}dz\int_{|y| < \sqrt{z^{2-\alpha}}}|L(y)|^{q}z^{(n-1)\alpha/2}dy;
\end{align*}
this implies (\ref{eq:epo29}).
With the same argument as above, we can prove the following estimates:
\begin{align*}
&cR^{(\alpha(n-1)+2\gamma)/2q} \leq  \|\Lambda\psi_{R}\psi^{'}\|_{L^{q}_{x}} \leq CR^{(\alpha(n-1)+2\gamma)/2q}
\\
&cR^{(\alpha(n-1)+2\gamma)/2q - \gamma} \leq  \|\Lambda\psi^{'}_{R}\psi\|_{L^{q}_{x}} \leq CR^{(\alpha(n-1)+2\gamma)/2q - \gamma}
\\
&cR^{(\alpha(n-1)+2\gamma)/2q - \gamma} \leq  \|\Lambda\psi^{'}_{R}\psi^{'}\|_{L^{q}_{x}} \leq CR^{(\alpha(n-1)+2\gamma)/2q - \gamma}
\\
&cR^{(\alpha(n-1)+2\gamma)/2q} \leq  \|\Lambda\psi_{R}\psi^{''}\|_{L^{q}_{x}} \leq CR^{(\alpha(n-1)+2\gamma)/2q}
\\
&cR^{(\alpha(n-1)+2\gamma)/2q - 2\gamma} \leq  \|\Lambda\psi^{''}_{R}\psi\|_{L^{q}_{x}} \leq CR^{(\alpha(n-1)+2\gamma)/2q - 2\gamma}
\end{align*}
Now we can pass to the
proof of \eqref{eq:epo23}. It is sufficient to use the estimates
\begin{equation}\label{eq:epo32}
\|\psi_{R}\psi F\|_{L^{p}_{T}L^{q}_{x}} \leq C T^{1/p} R^{(\alpha(n-1)+2\gamma)/2q}\max\{R^{-2}, T^{2}R^{-(2\alpha+2)}\},
\end{equation}
\begin{equation}\label{eq:epo33}
\|G_{R}\|_{L^{p}_{T}L^{q}_{x}} \leq C T^{1/p} R^{(\alpha(n-1)+2\gamma)/2q}\max\{R^{-2\gamma}, R^{-(6-2\alpha)}, TR^{-(\alpha+1+\gamma)}, TR^{-4}\},
\end{equation}
 which we are going to prove under the conditions $1/2 < \gamma < 1$ and $1 < \alpha < 2$. Looking at the structure of F, it is easy to see that, in order to deduce \eqref{eq:epo32}, it is sufficient to estimate the norms of functions of the following type:
\begin{equation}
\frac{t}{z^{\alpha+2}} \Lambda\psi_{R}\psi,
\qquad
\frac{t^{2}}{z^{2\alpha+2}} \Lambda\psi_{R}\psi,
\qquad
\frac{1}{z^{2}} \Lambda\psi_{R}\psi,
\end{equation}
where $\Lambda$ is defined by \eqref{eq:epo28} and L may change in different expressions. Let us consider the first term. Due to the properties of $\psi_{R}$, on the support of $\frac{t}{z^{\alpha+2}} \Lambda\psi_{R}\psi$ we have that $t/z^{\alpha+2} \leq Ct/R^{\alpha+2}$; hence we obtain
\begin{equation}\label{eq:epo35}
\|\frac{t}{z^{\alpha+2}} \Lambda\psi_{R}\psi\|_{L^{p}_{T}L^{q}_{x}} \leq CT^{1+1/p}\frac{1}{R^{\alpha+2}}\|\Lambda\psi_{R}\psi\|_{L^{q}_{x}} \leq  CT^{1+1/p}R^{(\alpha(n-1)+2\gamma)/2q-(\alpha+2)}
\end{equation}
where we have used \eqref{eq:epo29} in the last estimate. With the same arguments, we can deduce that
\begin{align*}
&\|\frac{t^{2}}{z^{2\alpha+2}} \Lambda\psi_{R}\psi\|_{L^{p}_{T}L^{q}_{x}} \leq CT^{2+1/p}R^{(\alpha(n-1)+2\gamma)/2q-(2\alpha+2)}
\\
&\|\frac{1}{z^{2}} \Lambda\psi_{R}\psi\|_{L^{p}_{T}L^{q}_{x}} \leq CT^{1/p}R^{(\alpha(n-1)+2\gamma)/2q - 2}.
\end{align*}
Notice that the norm estimate in \eqref{eq:epo35} is the geometric mean of the estimates above, so it cannot be the largest one; this remark proves  \eqref{eq:epo32}.

Analogously, looking at the structure of $G_{R}$ it is easy to see that the following estimates imply \eqref{eq:epo33}:
\begin{align*}
&\|\frac{1}{z^{2}} \Lambda\psi_{R}\psi^{'}\|_{L^{p}_{T}L^{q}_{x}} \leq CT^{1/p}R^{(\alpha(n-1)+2\gamma)/2q - 2},\\
&\|\frac{|y|^{2}}{z^{4}} \Lambda\psi_{R}\psi^{''}\|_{L^{p}_{T}L^{q}_{x}} = \|\frac{1}{z^{4-\alpha}}\Theta\psi_{R}\psi^{''}\|_{L^{p}_{T}L^{q}_{x}} \leq CT^{1/p}R^{(\alpha(n-1)+2\gamma)/2q - (4-\alpha)},\\
&\|\Lambda\psi^{''}_{R}\psi\|_{L^{p}_{T}L^{q}_{x}} \leq CT^{1/p}R^{(\alpha(n-1)+2\gamma)/2q - 2\gamma},\\
&\|\frac{|y|^{2}}{z^{3}}\Lambda\psi^{'}_{R}\psi^{'}\|_{L^{p}_{T}L^{q}_{x}} \leq CT^{1/p}R^{(\alpha(n-1)+2\gamma)/2q - (3-\alpha+\gamma)},\\
&\|\frac{|y|^{4}}{z^{6}} \Lambda\psi_{R}\psi^{''}\|_{L^{p}_{T}L^{q}_{x}} = \|\frac{1}{z^{6-2\alpha}}\Psi\psi_{R}\psi^{''}\|_{L^{p}_{T}L^{q}_{x}} \leq CT^{1/p}R^{(\alpha(n-1)+2\gamma)/2q - (6-2\alpha)},\\
&\|\frac{|y|^{2}}{z^{4}} \Lambda\psi_{R}\psi^{'}\|_{L^{p}_{T}L^{q}_{x}} = \|\frac{1}{z^{4-\alpha}}\Theta\psi_{R}\psi^{'}\|_{L^{p}_{T}L^{q}_{x}} \leq CT^{1/p}R^{(\alpha(n-1)+2\gamma)/2q - (4-\alpha)},\\
&\|\frac{t}{z^{\alpha+1}}\Lambda\psi^{'}_{R}\psi\|_{L^{p}_{T}L^{q}_{x}} \leq CT^{1+1/p}R^{(\alpha(n-1)+2\gamma)/2q - (\alpha+1+\gamma)},\\
&\|\frac{t|y|^{2}}{z^{\alpha+4}} \Lambda\psi_{R}\psi^{'}\|_{L^{p}_{T}L^{q}_{x}} = \|\frac{t}{z^{4}}\Theta\psi_{R}\psi^{'}\|_{L^{p}_{T}L^{q}_{x}} \leq CT^{1+1/p}R^{(\alpha(n-1)+2\gamma)/2q - 4},\\
&\|\frac{1}{z}\Lambda\psi^{'}_{R}\psi\|_{L^{p}_{T}L^{q}_{x}} \leq CT^{1/p}R^{(\alpha(n-1)+2\gamma)/2q - (1+\gamma)}.
\end{align*}
Here we denoted by
\begin{equation}\label{eq:tetapsi}
\Theta(y,z) := \left(\frac{|y|}{\sqrt{z^{\alpha}}}\right)^{2} L\left(\frac{y}{\sqrt{z^{\alpha}}}\right),
\quad
\Psi(y,z) := \left(\frac{|y|}{\sqrt{z^{\alpha}}}\right)^{4} L\left(\frac{y}{\sqrt{z^{\alpha}}}\right),
\end{equation}
for $(y,z) \in \mathbb{R}^{n-1} \times \mathbb{R}.$
In order to conclude the proof of the Lemma, it is now sufficient to remark that estimates
\eqref{eq:epo25} and \eqref{eq:epo26} follow from \eqref{eq:epo21}, \eqref{eq:epo22}, and \eqref{eq:epo23}, where we choose $T = R^{\beta}$.
\end{proof}

\section{Proof of Theorem \ref{thm:odd}}\label{sec:proofodd}
We can now prove Theorem \ref{thm:odd}. For any $x\in\R^n$, we denote by $x=(y,z)$, where $y=(y_1,\dots,y_{n-1})\in\R^{n-1}$, $z\in\R$. Since $M$ is anti-symmetric, we have that $\text{div}A\equiv0$ and the explicit expansion of $H$ in \eqref{eq:hamiltonianintr} is given by
\begin{equation}\label{eq:Hexpl}
  H=-\Delta+2iA\cdot\nabla+|A|^2.
\end{equation}
By homogeneity we have
\begin{equation}\label{eq:omog}
A(y,z) \cdot \nabla u =z^{1-\alpha}A\left(\frac{y}{z}, 1\right) \cdot \nabla u,
\end{equation}
for all $(y,z) \in \mathbb{R}^{n-1} \times (0, \infty)$. Let $P = (0,...,0,1)$;
since $A(P)=0$ and $DA(P)=M$, doing the first-order Taylor expansion of $A$ around $P$ in \eqref{eq:omog} we get
\begin{align}\label{eq:epo39}
2iA(y,z)\cdot\nabla u
& =
2iz^{1-\alpha}\left\{M\left(\frac{y}{z}, 0\right)^{t}
+
R_{1}\left(\frac{y}{z}\right)\right\}\cdot\nabla u
\\
& = 2iz^{-\alpha}M\left(y,0\right)^{t}\cdot\nabla u
+
2iz^{1-\alpha}R_{1}\left(\frac{y}{z}\right)\cdot\nabla u,
\nonumber
\end{align}
where the rest $R_1$ satisfies
\begin{equation}\label{eq:epo42}
\left|R_{1}\left(\frac{y}{z}\right)\right| \leq C\frac{|y|^{2}}{z^{2}},
\end{equation}
for all $(y,z)\in\R^{n-1}\times(0,+\infty)$ such that $|y|<|z|$.
Analogously, for $|A|^2$ we have
\begin{equation}\label{eq:omog2}
|A|^{2}(y,z) = z^{2-2\alpha}|A|^{2}\left(\frac{y}{z}, 1\right).
\end{equation}
Notice that $|A|^2(P)=0=\nabla(|A|^2)(P)$, and moreover
\begin{equation*}
  D^2(|A|^2)(P)=2M^tM
  =
  2\left(
  \begin{array}{cccc}
    I_{n-1} & 0
    \\
    0 & 0
  \end{array}
  \right),
\end{equation*}
where $I_{n-1}$ denote the identity $(n-1)\times(n-1)$-matrix.
Hence, doing the second-order Taylor expansion of $|A|^2$ around $P$ in \eqref{eq:omog2}, we obtain
\begin{equation}\label{eq:epo50}
  |A|^2(y,z)=\frac{2}{z^{2\alpha}}|y|^2 + z^{2-2\alpha}R_{2}\left(\frac{y}{z}\right),
\end{equation}
where the rest $R_2$ satisfies
\begin{equation}\label{eq:epo51}
|R_{2}\left(\frac{y}{z}\right)| \leq C\frac{|y|^{3}}{z^{3}},
\end{equation}
provided $|y| < |z|$.

We can now select a couple $(\lambda, v(y))$ which satisfies the eigenvalue problem \eqref{eq:epo2}; hence from now on the functions $W_{R}$, $f_{R}$, and $F_{R}$ are fixed by \eqref{eq:WR}, \eqref{eq:epo18} and \eqref{eq:epo19}.

Due to \eqref{eq:Hexpl}, \eqref{eq:epo39}, \eqref{eq:epo50}, the following auxiliar equation is naturally related to our Cauchy problem \eqref{eq:schromagnintr}
\begin{align}\label{eq:epo69}
& i\partial_{t}u_{R}-\left(\nabla- \frac{i}{z^{\alpha}} M(y,0)^t\right)^2 u_R
\\
&\ \ \ + \frac{1}{z^{2\alpha}}|y|^2u_R
+ 2i z^{1-\alpha}R_{1}\left(\frac{y}{z}\right) \nabla u_R + z^{2-2 \alpha}R_{2}\left(\frac{y}{z}\right)u_{R} =\tilde{F}_{R},
\nonumber
\end{align}
with initial datum
\begin{equation}\label{eq:epo69datum}
  u_{R}(0, y, z) = f_{R},
\end{equation}
where
\begin{equation}\label{eq:epo70}
 \tilde{F}_{R}(t,y,z) = \chi_{(0, R^{\beta})}(t)\left\{F_{R} + 2i z^{1-\alpha}R_{1}\left(\frac{y}{z}\right)\nabla W_{R} + z^{2-2 \alpha}R_{2}\left(\frac{y}{z}\right)W_{R}\right\},
\end{equation}
$\beta$ is the same as in Lemma \ref{lem:odd}, and $f_{R}$, $F_{R}$ are given by \eqref{eq:epo18}, \eqref{eq:epo19}. Notice that, due to \eqref{eq:epo17}, \eqref{eq:epo70}, the solution $u_R$ of the Cauchy problem \eqref{eq:epo69}-\eqref{eq:epo69datum} coincides with the solution $W_R$ of \eqref{eq:epo17} for small times $t\in(0,R^\beta)$. We prove the following crucial Lemma.
\begin{lemma}\label{lem:mainodd}
  Let $(p,q)$ be a Schr\"odinger admissible couple, $(p,q)\neq(\infty,2)$, and let $(p',q')$ be the dual couple. The following estimate holds:
  \begin{equation}\label{eq:epo71}
    \frac{\|W_{R}\|_{L^{p}((0,R^{\beta});L^{q}_{x})}}{\|\tilde{F}_{R}\|_{L^{p'}((0,R^{\beta});L^{q'}_{x})}} \geq CR^{\delta}, \qquad \forall 1/2 < \gamma < 1,
  \end{equation}
  where
  \begin{align}\label{eq:epo72}
     & \delta  =\delta(n, \alpha, \gamma, \beta, p) = 2\left(\frac{\beta n - (\alpha(n-1)+2\gamma)}{np}\right)\\
     &+ \min\left\{2\gamma - \beta, 2\alpha + 2 -3\beta, \frac{\alpha}{2} + 1 - \beta, 3 - \frac{\alpha}{2} - \beta, \gamma + 1 - \beta, \alpha + 2 - 2\beta\right\},
     \nonumber
  \end{align}
  and $\gamma$ is the same as in \eqref{eq:psiR}.
\end{lemma}
\begin{proof}
  Due to \eqref{eq:epo26}, we just need to estimate the rest terms in \eqref{eq:epo70}. For the term containing $R_2$, we have
\begin{align*}
  & \|z^{2-2\alpha}R_{2}\left(\frac{y}{z}\right)W_{R}\|_{L^{p'}((0, T);L^{q'}_{x})}
  \\
  &\ \ \
  \leq
  CT^{1/p'}\|\frac{|y|^{3}}{z^{2\alpha+1}}\omega\psi_{R}\psi\|_{L^{q'}_{x}}
  \leq
  CT^{1/p'}R^{-(\alpha/2+1)}\|M\psi_{R}\psi\|_{L^{q'}_{x}}
  \\
  &\ \ \ \leq  CT^{1/p'}R^{-(\alpha/2+1)+(\alpha(n-1)+2\gamma)/2q'},
\end{align*}
for any $R > 1$, and $1/2 < \gamma < 1$,
where $\omega$ is the rescaled eigenfunction in \eqref{eq:omega1}; here we denoted by
\begin{equation}\label{eq:MM}
  M(y,z) = \left(|y|/\sqrt{z^{\alpha}}\right)^{3} v\left(y/\sqrt{z^{\alpha}}\right),
\end{equation}
and we used \eqref{eq:epo29} with $L(y) = |y|^{3}v(y)$ at the last step.
Hence for $t\in(0,R^\beta)$ we obtain
\begin{equation}\label{eq:epo53}
\|z^{2-2\alpha}R_{2}\left(\frac{y}{z}\right)W_{R}\|_{L^{p'}((0, R^{\beta});L^{q'}_{x})} \leqslant CR^{\beta/p'-(\alpha/2+1)+(\alpha(n-1)+2\gamma)/2q'},
\end{equation}
for any $\beta > 0$.

Let us continue with the term corresponding to $R_{1}$ in \eqref{eq:epo70}. We need to estimate
\begin{equation}\label{eq:WR2}
\|z^{1-\alpha}R_{1}\left(\frac{y}{z}\right)\nabla W_{R}\|_{L^{p'}((0, T);L^{q'}_{x})}
\end{equation}
Notice that, by \eqref{eq:WR} we have
\begin{equation}\label{eq:WR3}
\nabla W_{R} = W\psi\nabla\psi_{R} + W\psi_{R}\nabla \psi+(\nabla W)\psi_{R}\psi;
\end{equation}
hence we can treat separately the three terms in \eqref{eq:WR2}. First we estimate
\begin{align*}\label{eq:epo54}
&\|z^{1-\alpha}R_{1}\left(\frac{y}{z} \right)W\psi\nabla\psi_{R}\|_{L^{p'}((0, T);L^{q'}_{x})}
\\
&\ \ \ \leq CT^{1/p'}\|\frac{|y|^{2}}{z^{\alpha+1}}\omega\psi^{'}_{R}\psi\|_{L^{q'}_{x}}
\leq CT^{1/p'}R^{-1}\|\Theta\psi^{'}_{R}\psi\|_{L^{q'}_{x}}
\\
&\ \ \ \leq CT^{1/p'}R^{-(\gamma+1)+(\alpha(n-1)+2\gamma)/2q'},
\end{align*}
for any $R > 1$, and $1/2 < \gamma < 1$, where $\Theta$ is the same as in \eqref{eq:tetapsi}.
Therefore
\begin{equation}\label{eq:epo55}
\|z^{1-\alpha}R_{1}\left(\frac{y}{z} \right)W\psi\nabla\psi_{R}\|_{L^{p'}((0, R^{\beta});L^{q'}_{x})} \leqslant CR^{\beta/p'-(\gamma+1)+(\alpha(n-1)+2\gamma)/2q'},
\end{equation}
for any $\beta > 0$. Analogously, since
\begin{equation*}
|\nabla \psi| \leqslant C\frac{|y|}{z^{2}}\psi^{'}\quad\text{provided that}\quad|y| < |z|,
\end{equation*}
we estimate
\begin{align*}
& \|z^{1-\alpha}R_{1}\left(\frac{y}{z} \right)W\psi_{R}\nabla \psi\|_{L^{p'}((0, T);L^{q'}_{x})}
\\
&\ \ \ \leq CT^{1/p'}\|\frac{|y|^{3}}{z^{\alpha+3}}\omega\psi_{R}\psi^{'}\|_{L^{q'}_{x}}
\leq CT^{1/p'}R^{-(3-\alpha/2)}\|M\psi_{R}\psi^{'}\|_{L^{q'}_{x}}
\\
&\ \ \ \leq  CT^{1/p'}R^{-(3-\alpha/2)+(\alpha(n-1)+2\gamma)/2q'},
\end{align*}
for any $R > 1$, and $1/2 < \gamma < 1$.
Hence
\begin{equation}\label{eq:epo57}
\|z^{1-\alpha}R_{1}\left(\frac{y}{z} \right)W\psi_{R}\nabla \psi\|_{L^{p'}((0, R^{\beta});L^{q'}_{x})} \leqslant CR^{\beta/p'-(3-\alpha/2)+(\alpha(n-1)+2\gamma)/2q'},
\end{equation}
for any $\beta > 0$.
For the remaining term, notice that by \eqref{eq:epo5} we have
\begin{align}\label{eq:epo58}
  & \|z^{1-\alpha}R_{1}\left(\frac{y}{z}\right)\psi_{R}\psi\nabla W\|_{L^{p'}((0, T);L^{q'}_{x})}
  \\
  &\
  \leq
  \|z^{1-\alpha}\frac{t}{z^{\alpha+1}}R_{1}
  \left(\frac{y}{z}\right)\omega\psi_{R}\psi\|_{L^{p'}((0, T);L^{q'}_{x})}
  +
  \|z^{1-\alpha}R_{1}\left(\frac{y}{z}\right)\psi_{R}\psi\nabla \omega\|_{L^{p'}((0, T);L^{q'}_{x})}.
  \nonumber
\end{align}
For the first term in \eqref{eq:epo58} we have
\begin{align*}\label{eq:epo60}
& \|z^{1-\alpha}\frac{t}{z^{\alpha+1}}R_{1}\left(\frac{y}{z}\right)
\omega\psi_{R}\psi\|_{L^{p'}((0, T);L^{q'}_{x})}
\\
&\ \ \ \leq CT^{1+1/p'}\|\frac{|y|^{2}}{z^{2\alpha+2}}\omega\psi_{R}\psi\|_{L^{q'}_{x}}
\leq CT^{1+1/p'}R^{-(\alpha+2)}\|\Theta\psi_{R}\psi\|_{L^{q'}_{x}}
\\
& \ \ \ \leq CT^{1+1/p'}R^{-(\alpha+2)+(\alpha(n-1)+2\gamma)/2q'},
\end{align*}
for any $R > 1$, and $1/2 < \gamma < 1$. Therefore
\begin{equation}\label{eq:1}
  \|z^{1-\alpha}\frac{t}{z^{\alpha+1}}R_{1}\left(\frac{y}{z}\right)
  \omega\psi_{R}\psi\|_{L^{p'}((0, R^\beta);L^{q'}_{x})}
  \leq
  CR^{\beta+\frac\beta{p'}-(\alpha+2)+(\alpha(n-1)+2\gamma)/2q'}
\end{equation}
for any $\beta > 0$. For the second term in \eqref{eq:epo58}, we proceed as follows:
\begin{align}\label{eq:epo61}
& \|z^{1-\alpha}R_{1}\left(\frac{y}{z}\right)\psi_{R}\psi\nabla \omega\|_{L^{p'}((0, T);L^{q'}_{x})}
\\
&\ \ \ \leq CT^{1/p'}\|\frac{|y|^{2}}{z^{\alpha+1}}\psi_{R}\psi\nabla \omega\|_{L^{q'}_{x}}
\leq CT^{1/p'}\|\frac{|y|^{2}}{z^{3\alpha/2+1}}\psi_{R}\psi(\nabla v)\left(\frac{y}{\sqrt{z^{\alpha}}}\right)\|_{L^{q'}_{x}}
\nonumber
\\
&\ \ \ \leq CT^{1/p'}R^{-(\alpha/2+1)}\|\widetilde{M}\psi_{R}\psi\|_{L^{q'}_{x}},
\nonumber
\end{align}
where
\begin{equation*}
  \widetilde{M}(y,z) =  \left(|y|/\sqrt{z^{\alpha}}\right)^{2} (\nabla v)\left(y/\sqrt{z^{\alpha}}\right).
\end{equation*}
Explicitly we have
\begin{align*}\label{eq:epo62}
\|\widetilde{M}\psi_{R}\psi\|_{L^{q'}_{x}}^{q'}
 & = \int_{R-R^{\gamma}}^{R+R^{\gamma}}dz\int_{|y| < z}\left| \left(\frac{|y|}{\sqrt{z^{\alpha}}}\right)^{2} (\nabla v)\left(\frac{y}{\sqrt{z^{\alpha}}}\right)\right|^{q'}dy
 \\
 & = \int_{R-R^{\gamma}}^{R+R^{\gamma}}dz\int_{|y| < \sqrt{z^{2-\alpha}}}\left||y|^{2}(\nabla v)(y)\right|^{q'}z^{(n-1)\alpha/2}dy
 \\
 &\leq CR^{(\alpha(n-1)+2\gamma)/2},
\end{align*}
for all $R>1$, and $1/2<\gamma<1$. Hence by \eqref{eq:epo61} we obtain
\begin{equation}\label{eq:2}
  \|z^{1-\alpha}R_{1}\left(\frac{y}{z}\right)\psi_{R}\psi\nabla \omega\|_{L^{p'}((0, R^\beta);L^{q'}_{x})}
  \leq
  CR^{\frac{\beta}{p'}-(\alpha/2+1)+\frac{\alpha(n-1)+2\gamma}{2q'}},
\end{equation}
for any $\beta>0$. In conclusion, by \eqref{eq:epo58}, \eqref{eq:1} and \eqref{eq:2} we obtain
\begin{align}\label{eq:epo65}
& \|z^{1-\alpha}R_{1}\left(\frac{y}{z}\right)\psi_{R}\psi\nabla W\|_{L^{p'}((0, R^{\beta});L^{q'}_{x})}
\\
&\ \
\leq CR^{\beta/p'+(\alpha(n-1)+2\gamma)/2q'}\max\{R^{\beta-(\alpha+2)}, R^{-(\alpha/2+1)}\},
\nonumber
\end{align}
for any $\beta>0$. Hence, by \eqref{eq:WR3}, \eqref{eq:epo55}, \eqref{eq:epo57} and \eqref{eq:epo65} we can conclude that
\begin{align}\label{eq:epo66}
&\|z^{1-\alpha}R_{1}\left(\frac{y}{z}\right)\nabla W_{R}\|_{L^{p'}((0, R^{\beta});L^{q'}_{x})}
\\
&\ \ \ \leq  CR^{\beta/p'+(\alpha(n-1)+2\gamma)/2q'}\max \{R^{-(\gamma+1)}, R^{\alpha/2-3},R^{\beta-(\alpha+2)}, R^{-(\alpha/2+1)}\},
\nonumber
\end{align}
for any $\beta > 0$.
Now, by \eqref{eq:epo53}, \eqref{eq:epo66} we get
\begin{align}\label{eq:epo67}
&\|z^{2-2\alpha}R_{2}\left(\frac{y}{z}\right)W_{R} + z^{1-\alpha}R_{1}\left(\frac{y}{z}\right)\nabla W_{R} \|_{L^{p'}((0, R^{\beta});L^{q'}_{x})}
\\
&\ \ \ \leq  CR^{\beta/p'+(\alpha(n-1)+2\gamma)/2q'}\max \{R^{-(\gamma+1)}, R^{\alpha/2-3},R^{\beta-(\alpha+2)}, R^{-(\alpha/2+1)}\},
\nonumber
\end{align}
for any $\beta > 0$.
By \eqref{eq:epo70}, we have
\begin{align}\label{eq:quasi}
  & \frac{\|W_{R}\|_{L^{p}((0,R^{\beta});L^{q}_{x})}}{\|\tilde{F}_{R}\|_{L^{p'}
  ((0,R^{\beta});L^{q'}_{x})}}
  \\
  &\ \
  \geq
  \frac{\|W_{R}\|_{L^{p}((0,R^{\beta});L^{q}_{x})}}{\|F_{R}\|_{L^{p'}
  ((0,R^{\beta});L^{q'}_{x})}+\|z^{2-2\alpha}
  R_{2}\left(\frac{y}{z}\right)W_{R} + z^{1-\alpha}R_{1}\left(\frac{y}{z}\right)\nabla W_{R} \|_{L^{p'}((0, R^{\beta});L^{q'}_{x})}}.
  \nonumber
\end{align}
Finally, the thesis \eqref{eq:epo71} follows from \eqref{eq:epo22}, \eqref{eq:epo23}, \eqref{eq:epo67} and \eqref{eq:quasi}, when the couple $(p,q)$ satisfies the Schr\"odinger admissibility condition.
\end{proof}
Now let us come back to the inhomogeneous Cauchy problem \eqref{eq:epo69}-\eqref{eq:epo69datum}. Notice that
\begin{equation}\label{eq:epo74}
 \|u_{R}\|_{L^{p}_{t}L^{q}_{x}} \geq \|u_{R}\|_{L^{p}((0, R^{\beta}); L^{q}_{x})} =  \|W_{R}\|_{L^{p}((0, R^{\beta}); L^{q}_{x})},
\end{equation}
from which it follows that
\begin{equation}\label{eq:epo75}
 \frac{\|u_{R}\|_{L^{p}_{t}L^{q}_{x}}}{\|f_{R}\|_{L^{2}_{x}} + \|\tilde{F}_{R}\|_{L^{p'}_{t}L^{q'}_{x}}} \geq \frac{ \|W_{R}\|_{L^{p}((0, R^{\beta}); L^{q}_{x})}}{\|f_{R}\|_{L^{2}_{x}} + \|\tilde{F}_{R}\|_{L^{p'}((0, R^{\beta}); L^{q'}_{x})}}
\end{equation} \medskip
\noindent
Notice that \eqref{eq:epo25} implies that for any $1 \leqslant p < \infty, \quad 1/2 < \gamma < 1$,
\begin{equation}\label{eq:epo76}
\frac{\|W_{R}\|_{L^{p}((0, R^{\beta}); L^{q}_{x})}}{\|f_{R}\|_{L^{2}_{x}}} \to +\infty,
\end{equation}
as $R\to+\infty$,
provided that $\beta > (\alpha(n-1)+2\gamma)/n$. On the other hand, the function $\delta(n, \gamma, \beta, p)$ defined in \eqref{eq:epo72}, varies continuously with $\alpha$, and in particular
\begin{align}\label{eq:epo77}
&\delta\left(n, \gamma, \frac{\alpha(n-1)+2\gamma}{n}, p\right)
\\
& =\min\left\{\frac{(n - 1)(2\gamma - \alpha)}{n}, \frac{(2 - \alpha)n + 3\alpha - 6\gamma}{n}, \frac{(2 - \alpha)n + 2\alpha - 4\gamma}{2n},
\right.
\nonumber
\\
&\ \ \ \ \left. \frac{3(2 - \alpha)n + 2\alpha - 4\gamma}{2n}, \frac{n(\gamma + 1) - (n - 1)\alpha - 2\gamma}{n}, \frac{(2 - \alpha)n + 2\alpha - 4\gamma}{n}\right\}.
\nonumber
\end{align}
 Hence $\delta$ is strictly positive if
 \begin{equation}\label{eq:range}
   \frac\alpha2 < \gamma < \frac\alpha2 + \frac{(2 - \alpha)n}6.
 \end{equation}
 Since $1<\alpha<2$, the range given by \eqref{eq:range} contains some $\gamma \in (1/2, 1)$; hence, by choosing $\beta > (\alpha(n-1)+2\gamma)/n$, we obtain $\delta>0$.
This remark, together with \eqref{eq:epo71}, gives
\begin{equation}\label{eq:epo78}
\frac{\|W_{R}\|_{L^{p}((0,R^{\beta});L^{q}_{x})}}{\|\tilde{F}_{R}\|_{L^{p'}((0,R^{\beta});L^{q'}_{x})}} \to +\infty,
\end{equation}
as $R\to+\infty$, and by \eqref{eq:epo75}, \eqref{eq:epo76}, \eqref{eq:epo78} we conclude that
\begin{equation}\label{eq:epo79}
\frac{\|u_{R}\|_{L^{p}_{t}L^{q}_{x}}}{\|f_{R}\|_{L^{2}_{x}} + \|\tilde{F}_{R}\|_{L^{p'}_{t}L^{q'}_{x}}} \to +\infty,
\end{equation}
as $R\to+\infty$.
The last inequality shows that the following Strichartz estimates
\begin{equation}\label{eq:epo80}
\|u\|_{L^{p}_{t}L^{q}_{x}} \leqslant C\left(\|f_{R}\|_{L^{2}_{x}} +  \|F\|_{L^{p'}_{t}L^{q'}_{x}}\right)
\end{equation}
cannot be satisfied by solutions of the inhomogeneus Schr\"odinger equation
\begin{equation}\label{eq:epo81}
\begin{cases}
i\partial_{t}u-\Delta u + 2iA(x)\cdot \nabla u+|A|^{2}(x)u=F
\\
u(x,0)=f(x),
\end{cases}
\end{equation}
where the potential $A$ is given by \eqref{eq:A}. In order to disprove Strichartz estimates for the corresponding homogeneus Schr\"odinger equation
in the non-endpoint case $p > 2$, it is sufficient to apply a $TT^{*}$-argument, by using the standard Christ-Kiselev Lemma (see \cite{CK}). Finally, also the endpoint estimate $p=2$ fails; indeed, if it were true, by interpolation with the mass conservation (i.e. the $L^\infty L^2$-estimate) also the non-endpoint estimates would be satisfied. This completes the proof of Theorem \ref{thm:odd}.

\section{Preliminaries for the even-dimensional case}\label{sec:even}
Before the proof of Theorem \ref{thm:even}, we start with some preliminary computations, analogous to the ones performed in Section \ref{sec:odd}. Throughout this Section, $n=2k$ will be an even number, $n\geq4$. In the even dimension $n-2$, let us again consider an eigenfunction $v$ of the operator $T$ corresponding to an eigenvalue $\lambda$, as in \eqref{eq:epo2}, \eqref{eq:lp}.
By scaling $v$ we can create a function $\omega$ on $\R^n$ as follows:
\begin{equation}
\omega(x) := v\left(\frac{y}{\sqrt{|z|^{\alpha}}}\right),
\qquad
x:=(y,z) \in \mathbb{R}^{n-2} \times \mathbb{R}^{2} \setminus \{0\},
\end{equation}
where $\alpha$ is the exponent in \eqref{eq:AA}. It is easy to see that $\omega$ satisfies
\begin{equation}\label{eq:epe4}
-\left(\nabla- \frac{i}{|z|^{\alpha}}M(y,0,0)^t\right)^2\omega +  \frac{1}{|z|^{2\alpha}}|y|^2\omega = \frac{\lambda}{|z|^{\alpha}} \omega,
\end{equation}
for any $(y,z) \in \mathbb{R}^{n-2}\times \mathbb{R}^{2} \setminus \{0\}$, where $M$ is the matrix in \eqref{eq:Meven}.
Let us now introduce the time-dependent function
\begin{equation}\label{eq:epe5}
W(t,y,z) = e^{i( \lambda t /|z|^{\alpha})}\omega(y,z),
\end{equation}
for $(t,y,z) \in \mathbb{R} \times \mathbb{R}^{n-2} \times \mathbb{R}^{2}\setminus \{0\}$. As in the odd-dimensional case, a direct computation shows that $W$ solves
\begin{equation}\label{eq:epe10}
i\partial_{t}W - \left(\nabla- \frac{i}{|z|^{\alpha}}M(y,0,0)^t\right)^2 W
+  \frac{1}{|z|^{2\alpha}}|y|^2W = F,
\end{equation}
where $F$ is given by
\begin{align}\label{eq:Feven}
 F(t,y,z)
 =  & e^{i( \lambda t /|z|^{\alpha})}\left\{
 \left[\frac{\alpha^{2}\lambda^{2}t^{2}}{|z|^{2\alpha+2}}-\frac{\alpha^{2}i \lambda t}{|z|^{\alpha+2}}\right]
 v\left(\frac{y}{\sqrt{|z|^{\alpha}}} \right)
 \right.
 \\
 & \left. - G\left(\frac{y}{\sqrt{|z|^{\alpha}}}\right)\cdot\left[\frac{\alpha^{2}i \lambda t}{|z|^{\alpha+2}} - \frac{\alpha^{2}}{4|z|^{2}}\right] - \frac{\alpha^{2}}{4|z|^{2}} H\left(\frac{y}{\sqrt{|z|^{\alpha}}}\right)
 \right\},
 \nonumber
\end{align}
with
\begin{equation}\label{eq:epe7}
G(y) = y \cdot \nabla_{y}v(y),
\qquad
H(y) = yD^{2}v(y)\cdot y.
\end{equation}
Now we introduce a smooth real-valued cutoff $\psi\in C^{\infty}(\mathbb{R})$ satisfying
\begin{equation*}
  \psi(s) = 0
  \quad
  \text{for } |s| > 1,
  \qquad
  \psi(s) = 1
  \quad
  \text{for } |s| < 1/2.
\end{equation*}
Let us fix a parameter $\gamma \in (1/2,1)$ and denote by
\begin{equation}\label{eq:psiReven}
  \psi^{1}_{R}(z) := \psi\left(\frac{z_{1} - R}{R^{\gamma}}\right),
  \qquad
  \psi^{2}_{R}(z) := \psi\left(\frac{z_{2} - R}{R^{\gamma}}\right),
\end{equation}
for $z=(z_1,z_2)\in\R^2$.
We truncate $W$ as follows:
\begin{equation}\label{eq:WReven}
W_{R}(t,y,z) := W(t,y,z)\psi^{1}_{R}(z)\psi^{2}_{R}(z)
\psi\left(\frac{|y|^{2}}{|z|^{2}}\right).
\end{equation}
Again, by direct computations it turns out that $W_R$ solves the Cauchy problem
\begin{equation}\label{eq:epe17}
\begin{cases}
  i\partial_{t}W_{R}-\left(\nabla-\frac{i}{|z|^{\alpha}}M(y,0,0)^t\right)^2W_{R}
  + \frac{1}{|z|^{2\alpha}}|y|^2W_{R} = F_{R}
  \\
  W_{R}(0,y,z) = f_{R}(y,z).
\end{cases}
\end{equation}
Here the initial datum is given by
\begin{equation}\label{eq:epe18}
f_{R}(y,z) = \psi^{1}_{R}(z)\psi^{2}_{R}(z)
\psi\left(\frac{|y|^{2}}{|z|^{2}}\right)\omega(y,z),
\end{equation}
while $F_R$ is given by
\begin{equation}\label{eq:epe19}
F_{R}(t,y,z) =  \psi^{1}_{R}\psi^{2}_{R}\psi F + G_{R},
\end{equation}
where $F$ is defined in \eqref{eq:Feven} and $G_R$ has the explicit form
\begin{align}\label{eq:Geven}
G_{R}(t,y,z) = &e^{i( \lambda t /|z|^{\alpha})}\left\{-\omega\left[\frac{2(n - 2)}{|z|^{2}}\psi^{1}_{R}\psi^{2}_{R}\psi^{'} + \frac{4|y|^{2}}{|z|^{4}}\psi^{1}_{R}\psi^{2}_{R}(\psi''+\psi')
+ \psi^{2}_{R}\psi^{1''}_{R}\psi
\right.\right.
\\
& + \psi^{1}_{R}\psi^{2''}_{R}\psi  - \frac{4|y|^{2}}{|z|^{4}}\left(\psi^{1'}_{R}\psi^{2}_{R}\psi^{'}z_{1}
+\psi^{1}_{R}\psi^{2'}_{R}\psi^{'}z_{2}\right) + \frac{4|y|^{4}}{|z|^{6}}\psi^{1}_{R}\psi^{2}_{R}\psi^{''}
\nonumber
\\
& \left.
 -
\frac{2\alpha i\lambda t}{|z|^{\alpha+2}}\left(\psi^{2}_{R}\psi^{1'}_{R}\psi z_{1}
+\psi^{1}_{R}\psi^{2'}_{R}\psi z_{2}\right)
+ \frac{4\alpha i\lambda t|y|^{2}}{|z|^{\alpha+4}}\psi^{1}_{R}\psi^{2}_{R}\psi^{'}\right]
\nonumber
\\
&- G\left(\frac{y}{\sqrt{|z|^\alpha}}\right)\cdot\left[ \frac{4}{|z|^{2}}\psi^{1}_{R}\psi^{2}_{R}\psi^{'}
+ \frac{2\alpha |y|^{2}}{|z|^{4}}\psi^{1}_{R}\psi^{2}_{R}\psi^{'}
\right.
\nonumber
\\
&\left.\left. - \frac{\alpha}{|z|^{2}}\left(\psi^{1'}_{R}\psi^{2}_{R}\psi z_{1}
+
\psi^{1}_{R}\psi^{2'}_{R}\psi z_{2}\right)\right]\right\}.
\nonumber
\end{align}
The main result of this Section is the following.
\begin{lemma}\label{lem:even}
Let $n\geq4$ be an even number, $p, q \in(1,\infty)$, and $\gamma\in(1/2,1)$; then we have
\begin{equation}\label{eq:epe21}
\|f_{R}\|_{L^{2}_{x}} \leq C R^{(\alpha(n-2)+4\gamma)/4},
\end{equation}
\begin{equation}\label{eq:epe22}
\|W_{R}\|_{L^{p}_{T}L^{q}_{x}} \geq C T^{1/p} R^{(\alpha(n-2)+4\gamma)/2q},
\end{equation}
\begin{equation}\label{eq:epe23}
\|F_{R}\|_{L^{p}_{T}L^{q}_{x}} \leqslant C T^{1/p} R^{(\alpha(n-2)+4\gamma)/2q}\max\{R^{-2\gamma}, T^{2}R^{-(2\alpha+2)}\},
\end{equation}
for all $R>2$, $T>0$ and some constant $C=C(q,\gamma)>0$
In particular, if $(p,q)\neq(\infty,2)$ is a Schr\"odinger admissible couple
and $\beta > 0$, then the following estimates hold
\begin{equation}\label{eq:epe25}
\frac{\|W_{R}\|_{L^{p}((0,R^{\beta}); L^{q}_{x})}}{\|f_{R}\|_{L^{2}_{x}}} \geq CR^{(\beta n - (\alpha(n-2)+4\gamma))/np},
\end{equation}
\begin{equation}\label{eq:epe26}
\frac{\|W_{R}\|_{L^{p}((0,R^{\beta}); L^{q}_{x})}}{\|F_{R}\|_{L^{p'}((0,R^{\beta}); L^{q'}_{x})}} \geq CR^{\kappa},
\end{equation}
for any $R>2$,
where
\begin{equation}
\kappa = \kappa(n, \gamma, \beta, p) = 2\left(\frac{\beta n - (\alpha(n-2)+4\gamma)}{np}\right) +\min\{2\gamma - \beta, 2\alpha + 2 - 3\beta\}
\end{equation}
and the constants $C > 0$ do not depend on $R > 2$.
\end{lemma}
\begin{remark}\label{rem:noproof}
  Notice that the statement of Lemma \ref{lem:even} is almost the same as the one of Lemma \ref{lem:odd} in the odd-dimensional case. The only difference is in the numerology of the exponents, and it is due to the dimensional gap between the cutoffs in \eqref{eq:WR}, \eqref{eq:WReven}. Since the proof is completely analogous (with more terms to be estimated due to \eqref{eq:Geven} but all of the same type), we omit straightforward details.
\end{remark}

\section{Proof of the Theorem \ref{thm:even}}\label{sec:proofeven}
We can pass to the proof of Theorem \ref{thm:even}, which is completely analogous to the one of Theorem \ref{thm:odd}. For $x\in\R^n$, we will denote $x=(y,z)$, where $y=(y_1,\dots,y_{n-2})\in\R^{n-2}$ and $z=(z_1,z_2)\in\R^2$. Let us write again the expansion
\begin{equation}\label{eq:Hexpleven}
  H=-\Delta+2iA\cdot\nabla+|A|^2.
\end{equation}
By homogeneity we have
\begin{equation}
A(y,z) = |z|^{1-\alpha}A\left(\frac{y}{|z|}, 0, 1\right),
\end{equation}
for all $(y,z) \in \mathbb{R}^{n-2}\times \mathbb{R}^{2}\setminus \{0\}$.
Let us fix again the point $P(0,\dots,0,1)$, as an arbitrary direction in the degeneracy plane generated by the axes $z_1,z_2$. By Taylor expanding around $P$, we get the analogous to formulas \eqref{eq:epo39} and \eqref{eq:epo50}, i.e.
\begin{align}\label{eq:epe39}
2iA(y,z)\cdot\nabla u
& =
2i|z|^{1-\alpha}\left\{M\left(\frac{y}{z}, 0, 0\right)^{t}
+
R_{1}\left(\frac{y}{|z|}\right)\right\}\cdot\nabla u
\\
& = 2i|z|^{-\alpha}M\left(y, 0, 0\right)^{t}\cdot\nabla u
+
2i|z|^{1-\alpha}R_{1}\left(\frac{y}{|z|}\right)\cdot\nabla u,
\nonumber
\end{align}
\begin{equation}\label{eq:epe50}
  |A|^2(y,z)=\frac{2}{|z|^{2\alpha}}|y|^2 + |z|^{2-2\alpha}R_{2}\left(\frac{y}{|z|}\right),
\end{equation}
where now $M$ is the matrix in \eqref{eq:Meven} and the rests satisfy
\begin{equation}\label{eq:epe42}
\left|R_{1}\left(\frac{y}{|z|}\right)\right| \leq C\frac{|y|^{2}}{|z|^{2}},
\qquad
|R_{2}\left(\frac{y}{|z|}\right)| \leq C\frac{|y|^{3}}{|z|^{3}},
\end{equation}
provided $|y|<|z|$.

As in Section \ref{sec:proofodd}, we select a couple $(\lambda, v(y))$ solving the eigenvalue problem \eqref{eq:epo2}; again, this fix the functions $W_{R}$, $f_{R}$, and $F_{R}$ which we use in the sequel.
As in the odd dimensional case, due to \eqref{eq:Hexpleven}, \eqref{eq:epe39}, \eqref{eq:epe50}, it is natural to consider the following auxiliar equation
\begin{align}\label{eq:epe69}
& i\partial_{t}u_{R}-\left(\nabla- \frac{i}{|z|^{\alpha}} M(y,0,0)^t\right)^2 u_R
\\
&\ \ \ + \frac{1}{|z|^{2\alpha}}|y|^2u_R
+ 2i |z|^{1-\alpha}R_{1}\left(\frac{y}{|z|}\right) \nabla u_R + |z|^{2-2 \alpha}R_{2}\left(\frac{y}{|z|}\right)u_{R} =\tilde{F}_{R},
\nonumber
\end{align}
with initial datum
\begin{equation}\label{eq:epe69datum}
  u_{R}(0, y, z) = f_{R},
\end{equation}
where
\begin{equation}\label{eq:epe70}
 \tilde{F}_{R}(t,y,z) = \chi_{(0, R^{\beta})}(t)\left\{F_{R} + 2i |z|^{1-\alpha}R_{1}\left(\frac{y}{|z|}\right)\nabla W_{R} + |z|^{2-2 \alpha}R_{2}\left(\frac{y}{|z|}\right)W_{R}\right\},
\end{equation}
$\beta$ is the same as in Lemma \ref{lem:even}, and $f_{R}$, $F_{R}$ are given by \eqref{eq:epe18}, \eqref{eq:epe19}. Again, notice that, due to \eqref{eq:epe17}, \eqref{eq:epe70}, the solution $u_R$ of the Cauchy problem \eqref{eq:epe69}-\eqref{eq:epe69datum} coincides with the solution $W_R$ of \eqref{eq:epe17} for small times $t\in(0,R^\beta)$. The crucial result is the following.
\begin{lemma}\label{lem:maineven}
  Let $(p,q)$ be a Schr\"odinger admissible couple, $(p,q)\neq(\infty,2)$, and let $(p',q')$ be the dual couple. The following estimate holds:
  \begin{equation}\label{eq:epe71}
    \frac{\|W_{R}\|_{L^{p}((0,R^{\beta});L^{q}_{x})}}{\|\tilde{F}_{R}\|_{L^{p'}
    ((0,R^{\beta});L^{q'}_{x})}} \geq CR^{\delta}, \qquad \forall 1/2 < \gamma < 1,
  \end{equation}
  where
  \begin{align}\label{eq:epe72}
     & \delta  =\delta(n, \alpha, \gamma, \beta, p) = 2\left(\frac{\beta n - (\alpha(n-2)+4\gamma)}{np}\right)\\
     &+ \min\left\{2\gamma - \beta, 2\alpha + 2 -3\beta, \frac{\alpha}{2} + 1 - \beta, 3 - \frac{\alpha}{2} - \beta, \gamma + 1 - \beta, \alpha + 2 - 2\beta\right\},
     \nonumber
  \end{align}
  and $\gamma$ is the same as in \eqref{eq:psiReven}.
\end{lemma}
\begin{remark}
  Notice that the only difference in the statements of Lemmas \ref{lem:mainodd}, \ref{lem:maineven} is in the first terms of formulas \eqref{eq:epo72}, \eqref{eq:epe72} for the exponents $\delta$. This is due to the dimensional difference between the cutoffs functions in \eqref{eq:psiR} and \eqref{eq:psiReven}.
\end{remark}
\begin{proof}[Proof of Lemma \ref{lem:maineven}]
  The proof is completely analogous to the one of Lemma \ref{lem:mainodd}. Again, due to \eqref{eq:epe26} it is sufficient to estimate the rest terms in \eqref{eq:epe70}. With the same type of computations made in the odd-dimensional case, we get the analogous to estimate \eqref{eq:epo67}, which is now the following:
\begin{align}\label{eq:epe67}
  &\||z|^{2-2\alpha}R_{2}\left(\frac{y}{|z|}\right)W_{R} + |z|^{1-\alpha}R_{1}\left(\frac{y}{|z|}\right)\nabla W_{R} \|_{L^{p'}((0, R^{\beta});L^{q'}_{x})}
  \\
  &\ \ \ \leq  CR^{\beta/p'+(\alpha(n-2)+4\gamma)/2q'}\max \{R^{-(\gamma+1)}, R^{\alpha/2-3},R^{\beta-(\alpha+2)}, R^{-(\alpha/2+1)}\},
  \nonumber
\end{align}
for any $\beta > 0$. By \eqref{eq:epe70}, we have
\begin{align}\label{eq:quasieven}
  & \frac{\|W_{R}\|_{L^{p}((0,R^{\beta});L^{q}_{x})}}{\|\tilde{F}_{R}\|_{L^{p'}
  ((0,R^{\beta});L^{q'}_{x})}}
  \\
  &
  \geq
  \frac{\|W_{R}\|_{L^{p}((0,R^{\beta});L^{q}_{x})}}{\|F_{R}\|_{L^{p'}
  ((0,R^{\beta});L^{q'}_{x})}+\|\frac{|z|^2}{|z|^{2\alpha}}
  R_{2}\left(\frac{y}{|z|}\right)W_{R} + |z|^{1-\alpha}R_{1}\left(\frac{y}{|z|}\right)\nabla W_{R} \|_{L^{p'}((0, R^{\beta});L^{q'}_{x})}}.
  \nonumber
\end{align}
Finally, the thesis \eqref{eq:epe71} follows from \eqref{eq:epe22}, \eqref{eq:epe23}, \eqref{eq:epe67} and \eqref{eq:quasieven}, when the couple $(p,q)$ satisfies the Schr\"odinger admissibility condition.
\end{proof}
In order to conclude the proof of Theorem \ref{thm:even}, let us come back to the inhomogeneous Cauchy problem \eqref{eq:epe69}-\eqref{eq:epe69datum}. We have
\begin{equation}\label{eq:epe74}
 \|u_{R}\|_{L^{p}_{t}L^{q}_{x}} \geq \|u_{R}\|_{L^{p}((0, R^{\beta}); L^{q}_{x})} =  \|W_{R}\|_{L^{p}((0, R^{\beta}); L^{q}_{x})},
\end{equation}
from which it follows that
\begin{equation}\label{eq:epe75}
 \frac{\|u_{R}\|_{L^{p}_{t}L^{q}_{x}}}{\|f_{R}\|_{L^{2}_{x}} + \|\tilde{F}_{R}\|_{L^{p'}_{t}L^{q'}_{x}}} \geq \frac{ \|W_{R}\|_{L^{p}((0, R^{\beta}); L^{q}_{x})}}{\|f_{R}\|_{L^{2}_{x}} + \|\tilde{F}_{R}\|_{L^{p'}((0, R^{\beta}); L^{q'}_{x})}}
\end{equation} \medskip
\noindent
Notice that \eqref{eq:epe25} implies that for any $1 \leq p < \infty, \quad 1/2 < \gamma < 1$,
\begin{equation}\label{eq:epe76}
\frac{\|W_{R}\|_{L^{p}((0, R^{\beta}); L^{q}_{x})}}{\|f_{R}\|_{L^{2}_{x}}} \to +\infty,
\end{equation}
as $R\to+\infty$,
provided that $\beta > (\alpha(n-2)+4\gamma)/n$. On the other hand, the function $\delta(n, \gamma, \beta, p)$ defined in \eqref{eq:epe72}, varies continuously with $\alpha$, and in particular
\begin{align}\label{eq:epe77}
&\delta\left(n, \gamma, \frac{\alpha(n-2)+4\gamma}{n}, p\right)
\\
& =\min\left\{\frac{(n - 2)(2\gamma - \alpha)}{n}, \frac{(2 - \alpha)n + 6\alpha - 12\gamma}{n}, \frac{(2 - \alpha)n + 4\alpha - 8\gamma}{2n},
\right.
\nonumber
\\
&\ \ \ \ \left. \frac{3(2 - \alpha)n + 4\alpha - 8\gamma}{2n}, \frac{n(\gamma + 1) - (n - 2)\alpha - 4\gamma}{n}, \frac{(2 - \alpha)n + 4\alpha - 8\gamma}{n}\right\}.
\nonumber
\end{align}
 Hence $\delta$ is strictly positive if
 \begin{equation}\label{eq:rangeeven}
   \frac\alpha2 < \gamma < \frac\alpha2 + \frac{(2 - \alpha)n}{12}.
 \end{equation}
 Since $1<\alpha<2$, the range given by \eqref{eq:rangeeven} contains some $\gamma \in (1/2, 1)$; hence, by choosing $\beta > (\alpha(n-2)+4\gamma)/n$, we obtain $\delta>0$.
This remark, together with \eqref{eq:epe71}, gives
\begin{equation}\label{eq:epe78}
\frac{\|W_{R}\|_{L^{p}((0,R^{\beta});L^{q}_{x})}}{\|\tilde{F}_{R}\|_{L^{p'}((0,R^{\beta});L^{q'}_{x})}} \to +\infty,
\end{equation}
as $R\to+\infty$, and by \eqref{eq:epe75}, \eqref{eq:epe76}, \eqref{eq:epe78} we conclude that
\begin{equation}\label{eq:epe79}
\frac{\|u_{R}\|_{L^{p}_{t}L^{q}_{x}}}{\|f_{R}\|_{L^{2}_{x}} + \|\tilde{F}_{R}\|_{L^{p'}_{t}L^{q'}_{x}}} \to +\infty,
\end{equation}
as $R\to+\infty$.
The proof now proceeds exactly as in the odd case.

\end{document}